\documentclass[english,british, toc=bibliography, toc=index, hidelinks, DIV=12,abstract=true]{scrartcl}
\usepackage[T1]{fontenc}
\usepackage[latin9]{inputenc}

\usepackage{babel}
\usepackage{varioref}
\usepackage{prettyref}
\usepackage{float}
\usepackage{mathtools}
\usepackage{enumitem}
\usepackage{multirow}
\usepackage{amsmath}
\usepackage{amsthm}
\usepackage{amssymb}
\usepackage{graphicx}
\usepackage[numbers,longnamesfirst]{natbib}

\makeatletter

\floatstyle{ruled}
\newfloat{algorithm}{tbp}{loa}
\providecommand{\algorithmname}{Algorithm}
\floatname{algorithm}{\protect\algorithmname}

\newtheorem{definition}{Definition}%

\theoremstyle{plain}
\newtheorem{thm}{\protect\theoremname}
\theoremstyle{plain}
\newtheorem{prop}[thm]{\protect\propositionname}
\theoremstyle{rem}
\newtheorem{rem}[thm]{\protect\remarkname}
\theoremstyle{plain}

\theoremstyle{plain}
\newtheorem{lem}[thm]{\protect\lemmaname}

\def\documenttitle{Lost customer approximations of SOQN with application to RMFS}

\usepackage{mathalfa}
\usepackage{lscape}
\usepackage{subfigure}
\usepackage{relsize}
\usepackage{listings} 
\usepackage{booktabs} 
\usepackage{array}
\usepackage{siunitx}
\usepackage[titletoc,title]{appendix}
\usepackage{algorithm}
\usepackage{algpseudocode}
\usepackage{authblk} 

\usepackage[unicode=true,
bookmarks=true,bookmarksnumbered=false,bookmarksopen=false,
breaklinks=true,pdfborder={0 0 0},backref=false,colorlinks=false]   
{hyperref}

\hypersetup{
	linkcolor=blue,     
	citecolor=blue, urlcolor=blue, filecolor=blue,
	pdftitle={\documenttitle},
	pdfauthor={Sonja Otten},
	pdfkeywords={}
}

\newrefformat{def}{Definition \ref{#1}}
\newrefformat{rem}{Remark \ref{#1}}
\newrefformat{sect}{Section \ref{#1}}
\newrefformat{sub}{Section \ref{#1}}
\newrefformat{prop}{Proposition \ref{#1}}
\newrefformat{thm}{Theorem \ref{#1}}
\newrefformat{chap}{Chapter \ref{#1}}
\newrefformat{cor}{Corollary \ref{#1}}
\newrefformat{par}{Paragraph \ref{#1}}
\newrefformat{ex}{Example \ref{#1}}
\newrefformat{part}{Part \ref{#1}}
\newrefformat{fig}{Figure \ref{#1}}
\newrefformat{app}{Appendix \ref{#1}}

\makeatother

\addto\captionsbritish{\renewcommand{\corollaryname}{Corollary}}
\addto\captionsbritish{\renewcommand{\lemmaname}{Lemma}}
\addto\captionsbritish{\renewcommand{\propositionname}{Proposition}}
\addto\captionsbritish{\renewcommand{\remarkname}{Remark}}
\addto\captionsbritish{\renewcommand{\theoremname}{Theorem}}
\addto\captionsenglish{\renewcommand{\algorithmname}{Algorithm}}
\addto\captionsenglish{\renewcommand{\corollaryname}{Corollary}}
\addto\captionsenglish{\renewcommand{\lemmaname}{Lemma}}
\addto\captionsenglish{\renewcommand{\propositionname}{Proposition}}
\addto\captionsenglish{\renewcommand{\remarkname}{Remark}}
\addto\captionsenglish{\renewcommand{\theoremname}{Theorem}}
\providecommand{\corollaryname}{Corollary}
\providecommand{\lemmaname}{Lemma}
\providecommand{\propositionname}{Proposition}
\providecommand{\remarkname}{Remark}
\providecommand{\theoremname}{Theorem}

\begin{document}
	\global\long\def\myV{V}%
	\global\long\def\myv{v}%
	\global\long\def\evect{\mathbf{e}}%
	\global\long\def\maxdiff{\widetilde{\max}(\mathbf{k})}%
	\global\long\def\routingprob{p_{i}(\mathbf{k})}%
	
	\global\long\def\kvect{\mathbf{k}}%
	\global\long\def\mvect{\mathbf{m}}%
	\global\long\def\nvect{\mathbf{n}}%
	\global\long\def\pvect{\mathbf{p}}%
	\global\long\def\zvect{\mathbf{z}}%
	\global\long\def\bvect{\mathbf{b}}%
	\global\long\def\routep#1#2{r\left(#1,#2\right)}%
	\global\long\def\Mset{\overline{M}}%
	\global\long\def\Jset{\overline{J}}%
	\global\long\def\supplierrate{\nu}%
	\global\long\def\phantomeq{\mathrel{\phantom{=}}}%
	
\title{Stability of queueing-inventory systems \\ with different priorities}

\date{}
\renewcommand\Affilfont{\small \itshape}
\author[1]{Sonja Otten} 
\affil[1]{Hamburg University of Technology, Am Schwarzenberg-Campus 3, 21073 Hamburg}
\author[2]{Hans Daduna} 
\affil[2]{Universität Hamburg, Bundesstraße 55, 20146 Hamburg}
\renewcommand\Authands{, }

\maketitle
\vspace{-1.5cm}
\begin{abstract}
	\noindent We study a production-inventory system with two customer classes with different priorities
	which are admitted to the system following  a flexible admission control scheme.
	The inventory management is according to a base
	stock policy and arriving demand which finds the inventory depleted is lost (lost sales).
	We analyse the global balance equations of the associated Markov process and  derive structural properties of the steady state distribution which provide
	insights into the equilibrium behaviour of the system. 
	We derive a sufficient condition for ergodicity using the Foster-Lyapunov stability criterion. For a special case we show that the condition is necessary as well.
	
	\begingroup
	\renewcommand\thefootnote{}%
	\footnote{
		ORCID and email address:\\
		Sonja Otten \includegraphics[width=0.8em]{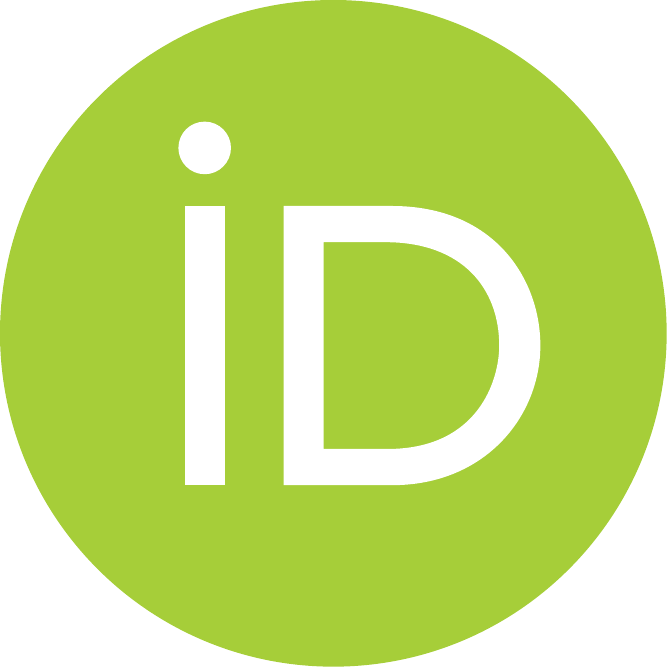}\hspace{0.4em}\url{https://orcid.org/0000-0002-3124-832X},
		sonja.otten@tuhh.de,\\
		Hans Daduna \includegraphics[width=0.8em]{orcid_icon}\hspace{0.4em}\url{https://orcid.org/0000-0001-6570-3012}, hans.daduna@uni-hamburg.de

	}%

	\addtocounter{footnote}{-1}%
	\endgroup
	
\end{abstract}
\emph{MSC 2010 Subject Classification:} 60K25; 68M20; 90B22

\noindent \emph{Keywords:} queueing networks, inventory control, priority customer


\section{Introduction}
The research communicated in this article  is devoted to 
queueing-inventory systems where customers of two classes of different priorities are served and for serving a customer 
a piece of raw material from an associated finished goods inventory is needed.
The inventory management is according to a standard base stock policy and replenishment orders are  fulfilled by an external supplier and lead time is positive.  Customers arriving when the inventory is depleted are rejected (lost sales) and additionally arrivals are
regulated by a flexible admission control which respects priorities.

Our main interest are stability conditions for this system and the stationary behaviour of a stable system.
We derive  structural properties of the steady state distribution which provide
insights into the equilibrium behaviour of the system but an explicit expression of the
complete stationary distribution is still an open problem.
Inspection of the literature shows that up to now there is 
no result available which provides rigorously proved conditions for stability in case that both customer classes have unbounded queues.
Our main result is therefore to derive a sufficient condition for ergodicity by the Foster-Lyapunov
stability criterion. For a subclass of the admission policies we show that this condition is necessary as well.
We expect that this is in general not the case. Furthermore, we consider the  case of instant service, where we determine the stationary distribution explicitly.\\

The paper is organised as follows. 
In \prettyref{sec:literature}, we describe the related literature. 
In \prettyref{sec:model}, we introduce our integrated model for production and inventory management and provide a Markov process $Z$ for the description of the time evolution of the integrated system.
In \prettyref{sec:DC-bs-ergodicity-1}, ergodicity is investigated in detail.
In \prettyref{sec:DC-bs-properties}, we assume that the queueing-inventory
process $Z$ is ergodic to analyse the properties of the stationary system.  
In \prettyref{sec:DC-bs-pure-inventory}, we consider the case of zero service time.
In Section \ref{sect:Conclusion} we summarise our main findings and indicate further research directions.


\section{Related literature and own contributions}\label{sec:literature}
Our research is connected to various parts of queueing theory and inventory control and their interplay in queueing-inventory models.
Our fundamental
production system is a classical $M/M/1/\infty$ queueing system, see e.g.~Wolff \cite[Chapter 5]{wolff:89}.
The specific features added to the $M/M/1/\infty$ in this article are the following:
\begin{itemize}
	\item Service of a customer needs a piece of raw material available in an associated inventory (for a review of research on these models see e.g.~Krishnamoorthy and his coauthors \cite{krishnamoorthy;shajin;narayanan:21}).
	\item Arriving customers have different priorities which result in a complicated admission control problem (for fundamentals on priority queues see Jaiswal \cite{jaiswal:68}).
\end{itemize}

Although in our models the customers are of different priority classes,  they require the same type of raw material stored in a single inventory and they experience the same service time distributions. The first property means that the items of raw material in the inventory are exchangeable with respect to the requesting customers' services.
This property is common in certain applications of e.g.~maintenance and repair, see e.g.~investigations by Ravid, Boxma, and Perry~\cite{ravid;boxma;perry:13}
and Daduna~\cite{daduna:90} and references therein.\\
The second property is discussed and motivated by Wang and his coauthors~\cite[p. 733]{wang;baron;scheller-wolf:15} who identify three main motivations for prioritisation and argue that two of these lead to identical service times over classes.

\textit{Literature on inventory theory} related to our problems
encompass the following problems and articles.
Studies by Gruen, Corsten, and Bharadwaj~\cite{gruen;corsten;bharadwaj:2002}~and
Verhoef and Sloot~\cite{verhoef;sloot:2006} analyse customers' behaviour
in practice and show that in many retail settings most of the original
demand can be considered to be lost in case of a stockout.
For an overview of the literature on systems with lost sales we refer
to Bijvank and Vis~\cite{bijvank;vis:11}. They present a classification
scheme for the replenishment policies most often applied in literature
and practice, and they review the proposed replenishment policies, including the base stock policy.\\
Tempelmeier~\cite[p.\ 84]{tempelmeier:2005}
argued that base stock control is economically reasonable if the order
quantity is limited because of technical reasons. 
The base stock policy is ``(...)\ more suitable for item with low demand, including the case of most spare parts'' \cite[p. 661]{rego:11}.\\
Morse~\cite[p.\ 139]{morse:1958} investigated (pure) inventory systems
that operate under a base stock policy. He gives a very simple example
where the concept ``re-order for each item sold'' is useful: Items
in inventory are bulky, and expensive (automobiles or TV sets\footnote{The paper is from 1958.}).
He uses queueing theory to model the inventory systems, analogously to e.g.~Reed and Zhang \cite{reed;zhang:17}.\\
Using queueing-theoretical methods to solve inventory models
is well established, often under the heading of ``production-inventory systems''. In this setting the production usually refers to the replenishment systems, early references are e.g.~Morse \cite{morse:1958}, Karush \cite{karush:57} (lost sales), Kaplan \cite{kaplan:70} (backordering).\\
More recent examples are
Rubio and Wein~\cite{rubio;wein:96}, Lee and Zipkin \cite{lee;zipkin:92},
\cite{lee;zipkin:95}, Zazanis~\cite{zazanis:94}, and Li and Arrelo-Risa \cite{li;arreola-risa:21} who
investigated classical single item and multi-item inventory systems.
Especially they evaluate the performance of base stock control policies in complex situations.

For a review of literature on inventory control systems 
(= queueing-inventory systems with service time equal to $0$, i.e.~``instant service'' \cite{melikov;ponomarenko;aliyev:18}) with multiple
customer classes, we refer to Isotupa (\cite[Section 2, pp. 3ff.]{isotupa:11}, 
\cite[Section 1, pp. 411ff.]{isotupa:15}) and Arslan and his coauthors \cite[pp. 1486ff.]{arslan;graves;roemer:07}. It is  differentiated between priority disciplines that regulate
customer arrivals and priority disciplines that regulate customer services. We combine both features in our model.

\textit{Literature on integrated queueing-inventory models} 
seemingly appeared only from 1992 on, see e.g.~Sigman and Simchi-Levi \cite{sigman;simchi-levi:92}, Melikov and Molchanov \cite{melikov;molchanov:92}.
For a recent extensive review we refer to Krishnamoorthy, Shajin and Narayanan~\cite{krishnamoorthy;shajin;narayanan:21}.
Because the research described in the present article is on queueing-inventory systems with priority classes for customers we concentrate here on articles which deal with priority problems in queueing-inventory systems.\\
Importance of research on this topic is emphasized by e.g.~Yadavalli and his coauthors~\cite{yadavalli;anbazhagan;jeganathan:15} 
who remarked that 
patients with serious illnesses are given priority over patients opting for routine checks or else in multi-speciality hospitals. Similarly, Liu and  couauthors~\cite[pp. 1544f.]{liu;xi;chen:13}
stated that orders with long term contracts have higher priority than
unscheduled order since they may bear lower shortage cost than the booked orders.\\
To the best of our knowledge, queueing-inventory systems
(under the heading ``inventory in counter-stream serving systems'') with different classes of customers were first considered by Melikov and Fatalieva \cite{melikov;fatalieva:98} who formulated a Markov decision model to minimize a cost function which encompasses costs for waiting, inventory holding, loss of demand and for dispatching items. In this problem setting there is no direct prioritisation of classes.\\ 
Zhao and Lian~\cite{zhao;lian:11} seemingly were the first to investigate a system with Poisson arrivals of different priority classes, exponentially distributed
service and lead times under backordering. They found a priority service
rule to minimise the long-run expected waiting cost by a dynamic programming
method. They formulate the model as a level-dependent quasi-birth-and-death
process such that the steady state probability distribution of their production-inventory systems can be computed by the Bright-Taylor algorithm.\\
Liu and coauthors (\cite{cheng;zhou;liu:12},~\cite{liu;xi;chen:13},~\cite{liu;feng;wong:14})
introduce a flexible admission control with a priority parameter $0\leq p\leq1$
``for controlling the application of priority''~\cite[p. 181]{liu;feng;wong:14}.
The priority parameter $p$ indicates the probability with which the
arrivals of ordinary customers are treated like the arrivals of priority
customers. If $p=1$, there is no priority in regulating arrivals.
If $p=0$, there is a strict priority in regulating arrivals. In the
last case, their model is the same as the model of Isotupa \cite{isotupa:07sapna07}.
They derive the stationary distribution of the inventory levels and
some performance measures. To obtain the optimal inventory control
policy they construct a mixed integer optimization problem in~\cite{liu;xi;chen:13}
and~\cite{liu;feng;wong:14} and develop an efficient searching algorithm
in~\cite{cheng;zhou;liu:12}.\\
Jeganathan and coauthors investigate in a sequence of papers (e.g.~\cite{jeganathan;kathiresan;anbazhagan:16},~\cite{jeganathan:15},
\cite{yadavalli;anbazhagan;jeganathan:15}) queueing-inventory systems
with two classes of customers. They consider models with impatient
customers, an optional second service and a mixed priority service
(non-preemptive priority and preemptive priority).
\\
Li and Zhao~\cite{li;zhao:09} investigate a preemptive priority
queueing system (without inventory) with two classes of customers and an exponential single server
who serves the two classes of customers at potentially different rates.\\
Melikov and coauthors (\cite{melikov;ponomarenko;aliyev:18}, \cite{melikov;ponomarenko;aliyev:18a}) investigate variants of queueing-inventory systems with demand of two classes with high and low priority. They consider an admission control scheme which allows access for high priority customers in any case (backordering) and rejects low priority customers when the stock level is below a critical value.
Using approximation methods they evaluate various performance measures of the systems.\\

\textit{The problem of stability for queueing-inventory systems with priority classes of demand} has not found much interest in the literature.
Most of the papers with priorities in queueing-inventory systems use state space truncation, i.e.~consider finite waiting rooms for tackling the stability problem, either both customer queues are assumed to be finite (then stability is not an issue) or one
of the queues is truncated
(most often the low priority queue, but in some articles high priority customers never queue up).  Then the system fits into the realm of QBD processes. Some representative examples are\\
\textit{for finite state space:}
\cite{yadavalli;anbazhagan;jeganathan:15},  
\cite{liu;xi;chen:13}, \cite{cheng;zhou;liu:12} (service time $=0$), 
\cite{jeganathan;anbazhagan;kathiresan:13},
\cite{jeganathan;kathiresan;anbazhagan:16}, \cite{jeganathan:15},
\cite{melikov;ponomarenko;aliyev:18},
\cite{wang:15},  \\
\textit{for countable state space:}
\cite{shajin;dudin;dudina;krishnamoorthy:20} and
\cite{baek;dudina;kim:17} 
(only a queue for low priority customers),
\cite{melikov;ponomarenko;aliyev:18} and \cite{melikov;ponomarenko;aliyev:18a} (steady state distribution obtained with approximation methods),
\cite{zhao;lian:11}  (heuristic/intuitive criterion for ergodicity for two demand classes, extending the  single demand class case of \cite{schwarz;sauer;daduna;kulik;szekli:06}).\\

The conclusion is that (best to our knowledge) there exists up to now no rigoruos result (criterion) for stability of queueing-inventory systems with differentiated priorities for demand classes in case of state spaces which are
two-dimensional infinite, i.e.~including $\mathbb{N}_0^2$,
counting for high and low priority customers (demands).\\ 
This observation is not surprising because in a general context
the problem can be termed as ``ergodicity of random walks in the quarter plane'' ($\mathbb{N}_0^2$), which is known to be a notoriuosly hard problem, see Guy and his coauthors \cite{fayolle;iasnogorodski;malyshev:99}.
Even more, in the present problem setting this random walk is influenced by the additional (finite) dimension of the inventory. So our investigation is on ergodicity of random walks in the quarter plane in a random environment. A very recent investigation of such problems is from Dimitriou \cite{dimitriou:22}.\\


\noindent\textbf{Our main contributions} are the following:\\
Starting from  a Markovian description for a queueing-inventory
system with two unbounded queues for
customers of different priority classes and 
flexible admission control we derive (rigorously) sufficient conditions which guarantee stability of the system.
For special cases we even show that the condition is sharp (necessary for stability).
Although the complete stationary distribution seems to be an open problem we are able to derive several partial balance properties of the system which are of independent interest.
Furthermore, we consider the special case of zero
service time and determine the stationary distribution explicitly. \\


\noindent\textbf{Notations and Conventions:}
\begin{itemize}
	\item $\mathbb{N}:=\left\{ 1,2,3,\ldots\right\} $, $\mathbb{N}_{0}:=\{0\}\cup\mathbb{N}$.
	\item  The notation $\subset$ between sets means ``subset or equal'' and
	$\subsetneq$ means ``proper subset''.  For
	a set $A$ we denote by $\vert A\vert$ the number of elements in $A$.
	\item  $1_{\left\{ expression\right\} }$ is the indicator function which
	is $1$ if $expression$ is true and $0$ otherwise.
	\item  Empty sums are 0, and empty products are 1.
	\item  All random variables are
	defined on a common probability space $(\Omega,{\cal F},P)$.
	\item By Markov process we mean time-homogeneous continuous-time strong
	Markov process with discrete state space ($=$ Markov jump process).
	Markov processes occurring are assumed to be regular and having paths which are right-continuous with left limits (cadlag). A Markov process is regular if
	it is non-explosive (i.e.\ the sequence of jump times of the process
	diverges almost surely), its transition intensity matrix is conservative
	(i.e.\ row sums are $0$) and stable (i.e.\ all diagonal elements
	of the transition intensity matrix are finite).
\end{itemize}


\section{Description of the model}\label{sec:model}

The supply chain of interest is depicted in \prettyref{fig:DC-ARR-SER-figur-two-demand-classes}
and consists of two arrival streams of priority and ordinary customers, a production system
(a single server with two unlimited waiting rooms), an inventory and
a supplier.

\begin{figure}[h]
	\centering{}\includegraphics[width=1\columnwidth]{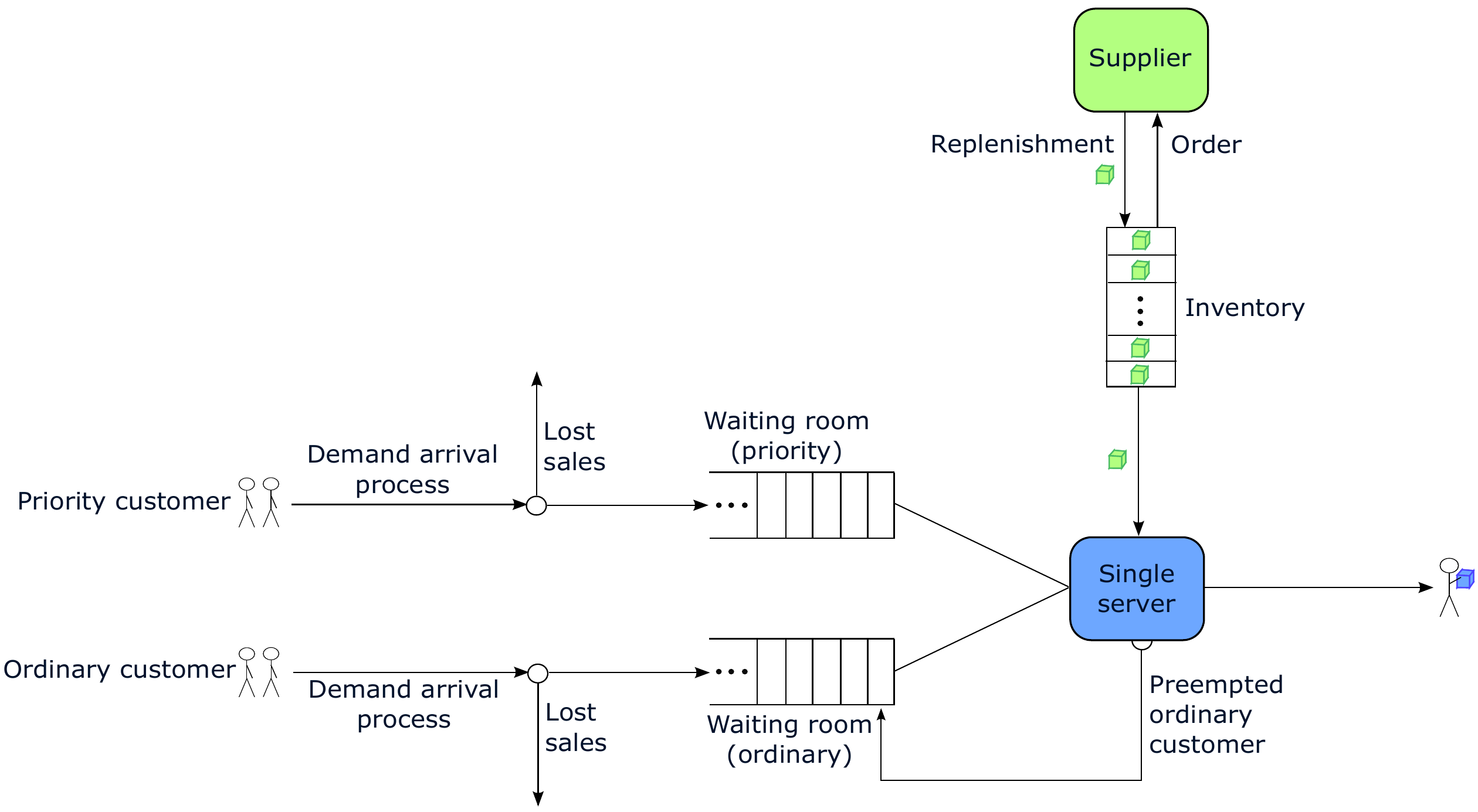}\caption{\label{fig:DC-ARR-SER-figur-two-demand-classes}The production-inventory
		system with two customer classes}
\end{figure}

The production system manufactures units according to customers' demand
on a make-to-order basis. There are two types of customers --- priority customers (type $1$) and ordinary customers (type $2$). $\overline{C}=\left\{ 1,2\right\} $
is the set of customer classes.
Priority customers arrive according
to a Poisson process with rate $\lambda_{1}>0$
and ordinary customers
arrive according to a Poisson process with rate $\lambda_{2}>0$.

Each customer needs exactly one item from the inventory for service.
The service time for both types of customers is exponentially distributed
with parameter $\mu>0$. If the
server is ready to serve a customer, who is at the head of the line,
and the inventory is not depleted, the service begins immediately.
Otherwise, the service starts at the instant of time when the next
replenishment arrives at the inventory.
A served customer departs from the system immediately and the associated
item is removed from the inventory at this time instant. 

An outside supplier replenishes raw material to the inventory according
to a base-stock  policy. Hence, each item taken from the inventory results in a direct order sent
to the supplier. This means, if a served customer departs from the
system, an order for one item of the consumed raw material is placed
at the supplier at this instant of time. The base stock level
$b\geq2$ is the maximal size of the
inventory. The replenishment lead time is exponentially
distributed with parameter $\nu>0$. (Note that there can be more than one outstanding order.)
Customers' arrivals are regulated by a flexible admission control
with priority parameter $p$, $0\leq p\leq1$:
If the inventory is depleted all arriving customers are rejected (``lost
sales''). If the on-hand inventory is greater than a prescribed threshold
level $s$, $0<s<b$, customers of both classes are admitted to enter the system. 
If the on-hand inventory reaches or falls below the threshold level
$s$, priority customers still enter the system but ordinary customers
are allowed to enter only with probability $p$ and are rejected with
probability $1-p$. 

There is a single server with two separate infinite waiting rooms
--- one waiting room for priority customers (priority queue) and
one waiting room for ordinary customers (ordinary queue) both under
a FCFS regime. If both customer queues are not empty, the server needs
to decide which one of them should be served. The choice is made according
to the preemptive resume discipline. An overview of various priority
disciplines can be found in~\cite[p. 53]{jaiswal:68}.

According to the preemptive resume discipline a newly arriving priority customer interrupts immediately
an ongoing service of an ordinary customer. The preempted ordinary
customer is put at the head of the ordinary customer queue and has to
wait until the priority queue is exhausted before he reenters service.
The preempted customer resumes service from the point of interruption
so that his service time upon reentry has been reduced by the amount
of time the customer has already spent in service (cf.~\cite[p. 1]{miller:58}).
Since it is assumed that the service time is exponential, the ordinary
customer requires on its reentry stochastically the same amount of
service as it required on its earlier entry. Thus, the preemptive
resume discipline is equal to the preemptive repeat-identical discipline
where the preempted customer requires the same amount of service on
its reentry as he required on his earlier entry (cf.~\cite[p. 53]{jaiswal:68}).

It is assumed that transmission times for orders are negligible and
set to zero and that the transportation time between the production system and the
inventory is negligible. All service times, inter-arrival times and replenishment lead times
constitute an independent family of random variables.\newpage

\textbf{A Markovian process} description of the integrated
queueing-inventory system is obtained as follows. Denote by $X_{1}(t)$  the number of priority customers present in the system at time $t\geq0$,
and by $X_{2}(t)$ the number of ordinary customers in the system at time $t\geq0$. (We call $X_{j}(t)$
the queue length = number of respective customers either waiting or in service.)
Since the customer in
service will always be of the priority class when at least one priority
customer is present, the value of the vector $(X_{1}(t),X_{2}(t))$
determines uniquely the type of the customer in service at time $t\geq0$,
if any. By $Y(t)$
we denote the on-hand inventory at time $t\geq0$.\\ 
We define  the joint queueing-inventory process of this system by
\[
Z=((X_{1}(t),X_{2}(t),Y(t)):t\geq0).
\]
Then, due to the usual independence and memoryless assumptions $Z$
is a homogeneous  Markov process. The state space of $Z$ is
\[
E=\left\{ (n_{1},n_{2},k):(n_{1},n_{2})\in\mathbb{N}_{0}^{2},\:k\in\{0,\ldots,b\}\right\} ,
\]
where $b$ is the maximal size of the inventory.
\label{DC-ARR-SER-bs-irreducible}
$Z$ has an infinitesimal generator $\mathbf{Q}=\left(q(z;\tilde{z}):z,\tilde{z}\in E\right)$
with the following  transition rates for $(n_{1},n_{2},k)\in E$:
\begin{align*}
q((n_{1},n_{2},k);(n_{1}+1,n_{2},k)) & =\lambda_{1}\cdot1_{\left\{ k>0\right\} },\\
q((n_{1},n_{2},k);(n_{1},n_{2}+1,k)) & =p\,\lambda_{2}\cdot1_{\left\{ 0<k\leq s\right\} }+\lambda_{2}\cdot1_{\left\{ k>s\right\} },\\
q((n_{1},n_{2},k);(n_{1}-1,n_{2},k-1)) & =\mu\cdot1_{\left\{ n_{1}>0\right\} }\cdot1_{\left\{ k>0\right\} },\\
q((n_{1},n_{2},k);(n_{1},n_{2}-1,k-1)) & =\mu\cdot1_{\left\{ n_{1}=0\right\} }\cdot1_{\left\{ n_{2}>0\right\} }\cdot1_{\left\{ k>0\right\} },\\
q((n_{1},n_{2},k);(n_{1},n_{2},k+1)) & =\nu\cdot1_{\left\{ k<b\right\} }.
\end{align*}
Furthermore, $q(z;\tilde{z})=0$ for any other pair $z\neq\tilde{z}$,
and
\[
q\left(z;z\right)=-\sum_{\substack{\tilde{z}\in E,\\
		z\neq\tilde{z}
	}
}q\left(z;\tilde{z}\right)\qquad\forall z\in E.
\]

\begin{definition}\label{defn:pi}
	If the queueing-inventory process $Z$ on state space $E$
	is ergodic we denote its 
	limiting and stationary distribution  by
	\begin{eqnarray*}
		\pi&:=&\left(\pi\left(n_{1},n_{2},k\right):\left(n_{1},n_{2},k\right)\in E\right)\\
		&&		\text{with}\quad \pi\left(n_{1},n_{2},k\right):=\lim_{t\rightarrow\infty}P\left(Z(t)=\left(n_{1},n_{2},k\right)\right).
	\end{eqnarray*}
\end{definition}

\noindent\textbf{Discussion of the modelling assumptions}

\textbf{(1)} The assumption of identical service time distributions for  customers of both classes is substantiated by 
Wang and  coauthors~\cite{wang;baron;scheller-wolf:15} 
who identify three main motivations for prioritisation:\\
$\bullet$ Different customers have  different willingness to pay for the same product.\\
$\bullet$ Customers may require different products or services, where some of these products are more profitable than others.\\
$\bullet$  Different service levels may substantially affect long-term profitability.\\ 
In our model we follow their conclusion: ``Modelling the effects of prioritization due to the first and third motivations can be achieved with identical service time distributions for different segments.''~\cite[p. 733]{wang;baron;scheller-wolf:15}

\textbf{(2)} Randomized differentiated admission control
is incorporated in many models and is discussed indepth 
by Isotupa (\cite{isotupa:06b},~\cite{isotupa:07sapna07},~\cite{isotupa:11},~\cite{isotupa;Samanta:13},~\cite{isotupa:15}) in a sequence of papers.
Isotupa investigates inventory systems with $(r,Q)$- and base stock policy for two classes of customers. 
Especially, the lost sales property for ordinary customers
is of interest, e.g.~for inventories
of spare parts in the airline or shipping industries (cf.~\cite[pp. 1f.]{isotupa:11}).\\
Beside of the standard lost sales rule (all arriving customers are rejected independent of class)
there is a modification such that
during the time the inventory is equal to or less
than a threshold level $s\leq Q$ 
arriving ordinary customers are lost. If the inventory level is zero,
demands due to both types of customers are lost, cf.~\cite{isotupa:07sapna07}, \cite{isotupa:11}, \cite{isotupa:15}.\\
Our admission control is a modification of this scheme such that the strict rejection of ordinary customers below the control limit $s$ is weakened by a randomized decision.

\textbf{(3)}
Preemptive resume regime  of priority customers over ordinary customers is a common scheme in queueing models, c.f.~\cite{jaiswal:68}. Coupling each service with obtaining an item from the inventory implies that an item dedicated at first to an ordinary customer can be attached to an interrupting priority customer as well. Here the exchangeble
property is necessary.

\section{Stability and stationary behaviour}\label{sect:StationaryBehaviour}
In this section, we investigate stability of the queueing-inventory system. This especially means to find conditions on the system which guarantee that the system approaches in the long run a steady state (i.e.~stabilises).
For queueing systems with priority classes of customers this is a non-trivial problem because in general (e.g.~in our setting of a queue without inventory) a solution of the global balance equations for the stationary distribution is not available and the standard ergodicity criterion for QBDs (Quasi-Birth-Death processes) is not applicable, see e.g.~\cite{latouche;ramaswami:99}\\
In Section \ref{sec:DC-bs-ergodicity-1} we find natural conditions for $Z$ to be ergodic by constructing a suitable Lyapunov function 
in the spirit of the classical Foster-Lyapunov stability criterion. 
Foster \cite{foster:57} introduced a technique for proving stability of Markov chains which was generalised in several directions. Our proof relies on Kelly and Yudovina \cite[Proposition D.3]{kelly;yudovina:14}.

\begin{prop}\label{prop:KellyYudovina14}
	Let $X:=(X(t):t\geq 0)$ be an irreducible  regular Markov process with countable state space $E$ and transition rate matrix $\mathbf{Q}:=(q(x,y):x,y\in E)$.
	Suppose that ${\cal L}:E\to [0,\infty)$ is a function such that for constants $\varepsilon>0$ and $b\in \mathbb{R}$, and some finite exception set $F\subset E$ and all $x\in E$ holds
	\begin{equation}
	\sum_{y\in E\setminus\{x\}}q(x,y)\left[{\cal L}(y)-{\cal L}(x)\right]\leq
	\begin{cases}
	-\varepsilon &x\notin F,\\
	b-\varepsilon & x\in F.
	\end{cases}
	\end{equation}
	Then $X$ is ergodic.
\end{prop}
Although we do not find a solution of the global balance equations for $Z$	
in Section \ref{sec:DC-bs-properties}, we present interesting results of the behaviour of $Z$ in steady state. These results rely on the presence of partial balance properties inherent in the dynamics of the queueing-inventory system in equilibrium.
Some of the resulting properties are surprising.\\

\subsection{Ergodicity\label{sec:DC-bs-ergodicity-1}\label{subsec:DC-bs-ergodic-foster}}

Irreducibility of the state process $Z$ can be seen directly from the transition rates at the end of  Section  \ref{sec:model}.
Our main result is the following theorem in the spirit of the Foster-Lyapunov stability criterion.
\begin{thm}
	\label{thm:DC-bs-foster-satz}The queueing-inventory process $Z$
	is ergodic if $\lambda_{1}+\lambda_{2}<\mu$.
\end{thm}
\begin{proof}
	Positive recurrence
	will be shown by the Foster-Lyapunov stability criterion 
	with $\mathcal{L}:E\rightarrow\mathbb{R}_{0}^{+}$ as Lyapunov function with
	\begin{equation}
	\mathcal{L}(n_{1},n_{2},k):=n_{1}+n_{2}+\alpha(k),
	\quad (n_1,n_2,k)\in E,		
	\label{eq:DC-bs-lypunovfunction}
	\end{equation}
	where $\alpha :\{0,1,\ldots,b\}\rightarrow [0,\infty)$ is strictly decreasing with
	\begin{align}
	\alpha(k)=(b-k)\cdot \frac{\mu-\lambda_{1}-\lambda_{2}}{2\mu} \label{eq:alpha}
	\end{align}
	and the finite exception set is given by
	\[
	F:=\left\{ (n_{1},n_{2},k):n_{1}+n_{2}=0\right\} .
	\]
	Furthermore, we define
	\begin{align}
	\eta :=\mu-\lambda_{1}-\lambda_{2} \label{eq:eta}
	\end{align}
	and
	\[\varepsilon :=\frac{\eta}{2}\cdot \min\left\{1, \frac{\nu}{\mu} \right\}.\]
	Due to the assumption, that $\mu>\lambda_1+\lambda_2$ it holds $\varepsilon>0$. \\
	$ $\\
	$\blacktriangleright$ First, we will check $\left(\mathbf{Q\cdot\mathcal{L}}\right)(n_{1},n_{2},k)<\infty$
	for $(n_{1},n_{2},k)\in F$.\\
	Since $0<\lambda_{1}<\infty$, $0<\lambda_{2}<\infty$ and $0<\nu<\infty$,\\
	for $k=0$ it holds
	\begin{align*}
	\left(\mathbf{Q}\cdot\mathcal{L}\right)(0,0,0) 
	& =\nu\cdot\left(\mathcal{L}(0,0,1)-\mathcal{L}(0,0,0)\right)
	\overset{\eqref{eq:DC-bs-lypunovfunction}}{=}\nu\cdot (\alpha(1)-\alpha(0))
	<\infty,
	\end{align*}
	%
	%
	for $k=1,\ldots,s$ it holds
	\begin{align*}
	\left(\mathbf{Q}\cdot\mathcal{L}\right)(0,0,k) 
	& =\lambda_{1}\cdot\left(\mathcal{L}(1,0,k)-\mathcal{L}(0,0,k)\right)+p\,\lambda_{2}\cdot\left(\mathcal{L}(0,1,k)-\mathcal{L}(0,0,k)\right)\\
	& \phantomeq+\nu\cdot\left(\mathcal{L}(0,0,k+1)-\mathcal{L}(0,0,k)\right)\\
	& \overset{\eqref{eq:DC-bs-lypunovfunction}}{=} \lambda_{1}\cdot [(1+\alpha(k))-\alpha(k)] 
	+p\,\lambda_{2}\cdot  [(1+\alpha(k))-\alpha(k)] \\
	&\phantomeq+ \nu\cdot (\alpha(k+1)-\alpha(k)) 
	<\infty,
	\end{align*}

	\noindent for $k=s+1,\ldots,b-1$ it holds
	\begin{align*}
	\left(\mathbf{Q}\cdot\mathcal{L}\right)(0,0,k) 
	& =\lambda_{1}\cdot\left(\mathcal{L}(1,0,k)-\mathcal{L}(0,0,k)\right)+\lambda_{2}\cdot\left(\mathcal{L}(0,1,k)-\mathcal{L}(0,0,k)\right)\\
	& \phantomeq+\nu\cdot\left(\mathcal{L}(0,0,k+1)-\mathcal{L}(0,0,k)\right)\\
	& \overset{\eqref{eq:DC-bs-lypunovfunction}}{=} \lambda_{1}\cdot [(1+\alpha(k))-\alpha(k)] 
	+\lambda_{2}\cdot  [(1+\alpha(k))-\alpha(k)] \\
	&\phantomeq+ \nu\cdot (\alpha(k+1)-\alpha(k)) 
	<\infty,
	\end{align*}
	
	\noindent for $k=b$ it holds
	\begin{align*}
	\left(\mathbf{Q}\cdot\mathcal{L}\right)(0,0,b) 
	& =\lambda_{1}\cdot\left(\mathcal{L}(1,0,b)-\mathcal{L}(0,0,b)\right)+\lambda_{2}\cdot\left(\mathcal{L}(0,1,b)-\mathcal{L}(0,0,b)\right)\\
	& \overset{\eqref{eq:DC-bs-lypunovfunction}}{=} \lambda_{1}\cdot [(1+\alpha(b))-\alpha(b)] 
	+\lambda_{2}\cdot  [(1+\alpha(b))-\alpha(b)] 
	<\infty.
	\end{align*}
	%
	%
	$ $\\
	$\blacktriangleright$ Second, we will check $\left(\mathbf{Q\cdot\mathcal{L}}\right)\left(n_{1},n_{2},k\right)\leq-\varepsilon$
	for $z=(n_{1},n_{2},k)\notin F$ with
	\begin{align}
	-\varepsilon & =-\frac{\eta}{2}\cdot \max\left\{1, \frac{\nu}{\mu} \right\}. \label{eq:minus-epsilon}
	\end{align}
	\\
	For $k=0$, $n_{1}\in \mathbb{N}$ and $n_{2}\in\mathbb{N}_{0}$ it holds
	\begin{align*}
	\left(\mathbf{Q}\cdot\mathcal{L}\right)(n_{1},n_{2},0) 
	& = \nu\cdot\left(\mathcal{L}(n_{1},n_{2},1)-\mathcal{L}(n_{1},n_{2},0)\right)\\
	& \overset{\eqref{eq:DC-bs-lypunovfunction}}{=}\nu\cdot\left((n_{1}+n_{2}+\alpha(1))-(n_{1}+n_{2}+\alpha(0))\right)
	=\nu\cdot (\alpha(1)-\alpha(0))\\
	&\overset{\eqref{eq:alpha}}{=}\nu \cdot[(b-1)-(b-0)]\cdot \frac{\mu-\lambda_{1}-\lambda_{2}}{2\mu}
	\overset{\eqref{eq:eta}}{=}-\frac{\eta}{2}\cdot\frac{\nu}{\mu}
	\overset{\eqref{eq:minus-epsilon}}{\leq}-\varepsilon\\
	\end{align*}
	%
	%
	For $k=0$, $n_{1}=0$ and $n_{2}\in \mathbb{N}$ it holds
	\begin{align*}
	\left(\mathbf{Q}\cdot\mathcal{L}\right)(0,n_{2},0) 
	& = \nu\cdot\left(\mathcal{L}(0,n_{2},1)-\mathcal{L}(0,n_{2},0)\right)\\
	&\overset{\eqref{eq:DC-bs-lypunovfunction}}{=}\nu\cdot\left((n_{2}+\alpha(1))-(n_{2}+\alpha(0))\right)
	=\nu\cdot (\alpha(1)-\alpha(0))\\
	&\overset{\eqref{eq:alpha}}{=}\nu \cdot[(b-1)-(b-0)]\cdot \frac{\mu-\lambda_{1}-\lambda_{2}}{2\mu}
	\overset{\eqref{eq:eta}}{=}-\frac{\eta}{2}\cdot\frac{\nu}{\mu}
	\overset{\eqref{eq:minus-epsilon}}{\leq}-\varepsilon\\
	\end{align*}
	%
	%
	\noindent For $k=1,\ldots,s$, $n_{1}\in \mathbb{N}$ and $n_{2}\in\mathbb{N}_{0}$ it holds
	\begin{align*}
	& \left(\mathbf{Q}\cdot\mathcal{L}\right)(n_{1},n_{2},k)\\
	& = \lambda_{1}\cdot\left(\mathcal{L}(n_{1}+1,n_{2},k)-\mathcal{L}(n_{1},n_{2},k)\right)
	+p\,\lambda_{2}\cdot\left(\mathcal{L}(n_{1},n_{2}+1,k)-\mathcal{L}(n_{1},n_{2},k)\right)\\
	& \phantomeq +\mu\cdot\left(\mathcal{L}(n_{1}-1,n_{2},k-1)-\mathcal{L}(n_{1},n_{2},k)\right)
	+\nu\cdot\left(\mathcal{L}(n_{1},n_{2},k+1)-\mathcal{L}(n_{1},n_{2},k)\right)\\
	& \overset{\eqref{eq:DC-bs-lypunovfunction}}{=}\lambda_{1}\cdot\left(n_{1}+1+n_{2}+\alpha(k)-n_{1}-n_{2}-\alpha(k)\right)\\
	& \phantomeq +p\,\lambda_{2}\cdot\left(n_{1}+n_{2}+1+\alpha(k)-n_{1}-n_{2}-\alpha(k)\right)\\
	& \phantomeq +\mu\cdot\left(n_{1}-1+n_{2}+\alpha(k-1)-n_{1}-n_{2}-\alpha(k)\right)\\
	& \phantomeq +\nu\cdot\left(n_{1}+n_{2}+\alpha(k+1)-n_{1}-n_{2}-\alpha(k)\right)\\
	& \overset{\substack{\eqref{eq:alpha}}}{=}\lambda_{1}+p\,\lambda_{2}-\mu
	+\mu\cdot\left[(b-(k-1))-(b-k)\right]\cdot \frac{\mu-\lambda_{1}-\lambda_{2}}{2\mu}\\
	&\hphantom{\overset{\substack{\eqref{eq:alpha}}}{=}}
	+\nu\cdot\left[(b-(k+1))-(b-k)\right]\cdot \frac{\mu-\lambda_{1}-\lambda_{2}}{2\mu}\\
	&= \lambda_{1}+p\,\lambda_{2}-\mu - \frac{\mu-\lambda_{1}-\lambda_{2}}{\mu} - \frac{\mu-\lambda_{1}-\lambda_{2}}{2\mu} \cdot \nu
	\overset{\substack{\eqref{eq:eta}}}{\leq} -\eta + \frac{\eta}{2} - \frac{\eta}{2}\cdot\frac{\nu}{\mu}
	\overset{\eqref{eq:minus-epsilon}}{\leq}-\varepsilon.\\
	\end{align*}
	%
	%
	For $k=1,\ldots,s$, $n_{1}=0$ and $n_{2}\in\mathbb{N}$ it holds
	\begin{align*}
	& \left(\mathbf{Q}\cdot\mathcal{L}\right)(0,n_{2},k)\\
	& =\lambda_{1}\cdot\left(\mathcal{L}(1,n_{2},k)-\mathcal{L}(0,n_{2},k)\right)
	+p\,\lambda_{2}\cdot\left(\mathcal{L}(0,n_{2}+1,k)-\mathcal{L}(0,n_{2},k)\right)\\
	& \phantomeq +\mu\cdot\left(\mathcal{L}(0,n_{2}-1,k-1)-\mathcal{L}(0,n_{2},k)\right)
	+\nu\cdot\left(\mathcal{L}(0,n_{2},k+1)-\mathcal{L}(0,n_{2},k)\right)\\
	& \overset{\eqref{eq:DC-bs-lypunovfunction}}{=}\lambda_{1}\cdot\left(1+n_{2}+\alpha(k)-n_{2}-\alpha(k)\right)
	+p\,\lambda_{2}\cdot\left(n_{2}+1+\alpha(k)-n_{2}-\alpha(k)\right)\\
	& \phantomeq +\mu\cdot\left(n_{2}-1+\alpha(k-1)-n_{2}-\alpha(k)\right)
	+\nu\cdot\left(n_{2}+\alpha(k+1)-n_{2}-\alpha(k)\right)\\
	& \overset{\substack{\hphantom{\eqref{eq:DC-bs-alphak1s}}}
	}{=}\lambda_{1}+p\,\lambda_{2}-\mu
	+\mu\cdot\left(\alpha(k-1)-\alpha(k)\right)+\nu\cdot\left(\alpha(k+1)-\alpha(k)\right)\\
	& \overset{\substack{\eqref{eq:alpha}}}{=}\lambda_{1}+p\,\lambda_{2}-\mu
	+\mu\cdot\left[(b-(k-1))-(b-k)\right]\cdot \frac{\mu-\lambda_{1}-\lambda_{2}}{2\mu}\\
	&\phantomeq +\nu\cdot\left[(b-(k+1))-(b-k)\right]\cdot \frac{\mu-\lambda_{1}-\lambda_{2}}{2\mu}\\
	&\overset{\substack{\eqref{eq:eta}}}{\leq} -\eta + \frac{\eta}{2} - \frac{\eta}{2}\cdot\frac{\nu}{\mu}
	\overset{\eqref{eq:minus-epsilon}}{\leq}-\varepsilon.\\
	\end{align*}
	%
	%
	\noindent For $k=s+1,\ldots,b-1$, $n_{1}\in\mathbb{N}$ and $n_{2}\in \mathbb{N}_0$ it holds
	\begin{align*}
	& \left(\mathbf{Q}\cdot\mathcal{L}\right)(n_{1},n_{2},k)\\
	& =\lambda_{1}\cdot\left(\mathcal{L}(n_{1}+1,n_{2},k)-\mathcal{L}(n_{1},n_{2},k)\right)
	+\lambda_{2}\cdot\left(\mathcal{L}(n_{1},n_{2}+1,k)-\mathcal{L}(n_{1},n_{2},k)\right)\\
	& \phantomeq +\mu\cdot\left(\mathcal{L}(n_{1}-1,n_{2},k-1)-\mathcal{L}(n_{1},n_{2},k)\right)
	+\nu\cdot\left(\mathcal{L}(n_{1},n_{2},k+1)-\mathcal{L}(n_{1},n_{2},k)\right)\\
	& \overset{\eqref{eq:DC-bs-lypunovfunction}}{=}\lambda_{1}\cdot\left(n_{1}+1+n_{2}+\alpha(k)-n_{1}-n_{2}-\alpha(k)\right)\\
	& \phantomeq +\lambda_{2}\cdot\left(n_{1}+n_{2}+1+\alpha(k)-n_{1}-n_{2}-\alpha(k)\right)\\
	& \phantomeq +\mu\cdot\left(n_{1}-1+n_{2}+\alpha(k-1)-n_{1}-n_{2}-\alpha(k)\right)\\
	& \phantomeq +\nu\cdot\left(n_{1}+n_{2}+\alpha(k+1)-n_{1}-n_{2}-\alpha(k)\right)\\
	& \overset{\substack{\hphantom{\eqref{eq:DC-bs-alphak1s}}}}{=}
	\lambda_{1}+\lambda_{2}-\mu
	+\mu\cdot\left(\alpha(k-1)-\alpha(k)\right)+\nu\cdot\left(\alpha(k+1)-\alpha(k)\right)\\
	& \overset{\substack{\eqref{eq:alpha}}}{=}\lambda_{1}+\lambda_{2}-\mu+\mu\cdot\left[(b-(k-1))-(b-k)\right]\frac{\mu-\lambda_{1}-\lambda_{2}}{2\mu}\\
	&\hphantom{\overset{\substack{\eqref{eq:alpha}}}{=}}
	+\nu\cdot\left[(b-(k+1))-(b-k)\right]\frac{\mu-\lambda_{1}-\lambda_{2}}{2\mu}\\
	&\overset{\substack{\eqref{eq:eta}}}{=} -\eta + \frac{\eta}{2} - \frac{\eta}{2}\cdot\frac{\nu}{\mu}
	\overset{\eqref{eq:minus-epsilon}}{\leq}-\varepsilon.\\
	\end{align*}
	%
	%
	For $k=s+1,\ldots,b-1$, $n_{1}=0$ and $n_{2}\in \mathbb{N}$
	\begin{align*}
	& \left(\mathbf{Q}\cdot\mathcal{L}\right)(0,n_{2},k)\\
	& =\lambda_{1}\cdot\left(\mathcal{L}(1,n_{2},k)-\mathcal{L}(0,n_{2},k)\right)
	+\lambda_{2}\cdot\left(\mathcal{L}(0,n_{2}+1,k)-\mathcal{L}(0,n_{2},k)\right)\\
	& \phantomeq +\mu\cdot\left(\mathcal{L}(0,n_{2}-1,k-1)-\mathcal{L}(0,n_{2},k)\right)
	+\nu\cdot\left(\mathcal{L}(0,n_{2},k+1)-\mathcal{L}(0,n_{2},k)\right)\\
	& \overset{\eqref{eq:DC-bs-lypunovfunction}}{=}
	\lambda_{1}\cdot\left(1+n_{2}+\alpha(k)-n_{2}-\alpha(k)\right)
	+\lambda_{2}\cdot\left(n_{2}+1+\alpha(k)-n_{2}-\alpha(k)\right)\\
	& \hphantom{\overset{\substack{\eqref{eq:DC-bs-alphak1s}}}{=}}
	+\mu\cdot\left(n_{2}-1+\alpha(k-1)-n_{2}-\alpha(k)\right)
	+\nu\cdot\left(n_{2}+\alpha(k+1)-n_{2}-\alpha(k)\right)\\
	& \overset{\substack{\hphantom{\eqref{eq:DC-bs-alphak1s}}}}{=}
	\lambda_{1}+\lambda_{2}-\mu+\mu\cdot\left(\alpha(k-1)-\alpha(k)\right)+\nu\cdot\left(\alpha(k+1)-\alpha(k)\right)\\
	& \overset{\substack{\eqref{eq:alpha}}
	}{=}\lambda_{1}+\lambda_{2}-\mu
	+\mu\cdot\left[(b-(k-1))-(b-k)\right]\frac{\mu-\lambda_{1}-\lambda_{2}}{2\mu}\\
	&\phantomeq +\nu\cdot\left[(b-(k+1))-(b-k)\right]\frac{\mu-\lambda_{1}-\lambda_{2}}{2\mu}\\
	&\overset{\substack{\eqref{eq:eta}}}{=} -\eta + \frac{\eta}{2} - \frac{\eta}{2}\cdot\frac{\nu}{\mu}
	\overset{\eqref{eq:minus-epsilon}}{\leq}-\varepsilon.\\
	\end{align*}
	%
	%
	\noindent For $k=b$, $n_{1}\in \mathbb{N}$ and $n_{2}\in\mathbb{N}_{0}$ it holds
	\begin{align*}
	& \left(\mathbf{Q}\cdot\mathcal{L}\right)(n_{1},n_{2},b)\\
	& = \lambda_{1}\cdot\left(\mathcal{L}(n_{1}+1,n_{2},b)-\mathcal{L}(n_{1},n_{2},b)\right)
	+\lambda_{2}\cdot\left(\mathcal{L}(n_{1},n_{2}+1,b)-\mathcal{L}(n_{1},n_{2},b)\right)\\
	& \phantomeq +\mu\cdot\left(\mathcal{L}(n_{1}-1,n_{2},b-1)-\mathcal{L}(n_{1},n_{2},b)\right)\\
	& \overset{\eqref{eq:DC-bs-lypunovfunction}}{=}\lambda_{1}\cdot\left(n_{1}+1+n_{2}+\alpha(b)-n_{1}-n_{2}-\alpha(b)\right)\\
	& \phantomeq +\lambda_{2}\cdot\left(n_{1}+n_{2}+1+\alpha(b)-n_{1}-n_{2}-\alpha(b)\right)\\
	& \phantomeq +\mu\cdot\left(n_{1}-1+n_{2}+\alpha(b-1)-n_{1}-n_{2}-\alpha(b)\right)\\
	& \overset{\substack{\eqref{eq:alpha}}}{=}
	\lambda_{1}+\lambda_{2}-\mu+\mu\cdot\left((b-(b-1))-(b-b)\right)\cdot \frac{\mu-\lambda_{1}-\lambda_{2}}{2\mu}\\
	&\overset{\substack{\eqref{eq:eta}}}{=}  -\eta + \frac{\eta}{2}
	\overset{\eqref{eq:minus-epsilon}}{\leq}-\varepsilon.\\
	\end{align*}
	%
	%
	For $k=b$, $n_{1}=0$ and $n_{2}\in\mathbb{N}$ it holds
	\begin{align*}
	& \left(\mathbf{Q}\cdot\mathcal{L}\right)(0,n_{2},b)\\
	& =\lambda_{1}\cdot\left(\mathcal{L}(1,n_{2},b)-\mathcal{L}(0,n_{2},b)\right)
	+\lambda_{2}\cdot\left(\mathcal{L}(0,n_{2}+1,b)-\mathcal{L}(0,n_{2},b)\right)\\
	& \phantomeq +\mu\cdot\left(\mathcal{L}(0,n_{2}-1,b-1)-\mathcal{L}(0,n_{2},b)\right)\\
	& \overset{\eqref{eq:DC-bs-lypunovfunction}}{=}\lambda_{1}\cdot\left(1+n_{2}+\alpha(b)-n_{2}-\alpha(b)\right)
	+\lambda_{2}\cdot\left(n_{2}+1+\alpha(b)-n_{2}-\alpha(b)\right)\\
	& \phantomeq +\mu\cdot\left(n_{2}-1+\alpha(b-1)-n_{2}-\alpha(b)\right)\\
	& \overset{\substack{\hphantom{\eqref{eq:DC-bs-alphak1s}}}
	}{=}\lambda_{1}+\lambda_{2}-\mu+\mu\cdot\left(\alpha(b-1)-\alpha(b)\right)\\
	& \overset{\substack{\eqref{eq:alpha}}}{=}
	\lambda_{1}+\lambda_{2}-\mu+\mu\cdot\left((b-(b-1))-(b-b)\right)\cdot \frac{\mu-\lambda_{1}-\lambda_{2}}{2\mu}\\
	&\overset{\substack{\eqref{eq:eta}}}{=} -\eta + \frac{\eta}{2}
	\overset{\eqref{eq:minus-epsilon}}{\leq}-\varepsilon.
	\end{align*}
\end{proof}

We do not expect that the result of the theorem is sharp for general $p\in[0,1]$, i.e.~the condition $\lambda_{1}+\lambda_{2}<\mu$ is in general only sufficient for stability. We have only 
\begin{thm}\label{thm:p=1}
	Consider the system from Theorem \ref{thm:DC-bs-foster-satz}
	for $p=1$, i.e.~low priority customers are admitted as long as the inventory is not depleted.
	Then $Z$ is ergodic if and only if $\lambda_{1}+\lambda_{2}<\mu$ holds.
\end{thm}
The result of Theorem \ref{thm:p=1} at a first glance seems to be intuitive because no distinction is made between the arriving customers. The problem is that the influence of the interrupts of arrivals and service due to stockout is not 
accessible to intuition. The Theorem states especially that the interrupts of arrivals and service balance over time.
The  proof of Theorem \ref{thm:p=1} is postponed to the next section because we need some preparations which we consider to be of independent interest.


\subsection{Properties of the stationary system\label{sec:DC-bs-properties}}

In this section, we assume that the queueing-inventory process $Z$ is ergodic. So the limiting and stationary distribution $\pi$ of $Z$ (see Definiton \ref{defn:pi}) exists but seems to be not accessible directly.
Nevertheless we are able to present results on the stationary behaviour of $Z$. We will exploit the structural information inherent in the global balanec equations of $Z$ for $\pi$.
\label{DC-ARR-SER-bs-irreducible}

The global balance equations $\pi\cdot\mathbf{Q=0}$ of the ergodic
queueing-inventory process $Z$ are for $(n_{1},n_{2},k)\in E$ given
by
\begin{align}
&\pi(n_{1},n_{2},k)\cdot\left((\lambda_{1}+p\,\lambda_{2})\cdot1_{\left\{ 0<k\leq s\right\} }+(\lambda_{1}+\lambda_{2})\cdot1_{\left\{ k>s\right\} }\right. \left.+\mu\cdot1_{\left\{ n_{1}+n_{2}>0\right\} }\cdot1_{\left\{ k>0\right\} }+\nu\cdot1_{\left\{ k<b\right\} }\right)\nonumber \\
& =\pi(n_{1}-1,n_{2},k)\cdot\lambda_{1}\cdot1_{\left\{ n_{1}>0\right\} }\cdot1_{\left\{ k>0\right\} }\nonumber \\
& \phantomeq+\pi(n_{1},n_{2}-1,k)\cdot p\,\lambda_{2}\cdot1_{\left\{ n_{2}>0\right\} }\cdot1_{\left\{ 0<k\leq s\right\} }\nonumber\\
& \phantomeq+\pi(n_{1},n_{2}-1,k)\cdot\lambda_{2}\cdot1_{\left\{ n_{2}>0\right\} }\cdot1_{\left\{ k>s\right\} }\nonumber \\
& \phantomeq+\pi(n_{1}+1,n_{2},k+1)\cdot\mu\cdot1_{\left\{ k<b\right\} }+\pi(n_{1},n_{2}+1,k+1)\cdot\mu\cdot1_{\left\{ n_{1}=0\right\} }\cdot1_{\left\{ k<b\right\} }\nonumber \\
& \phantomeq+\pi(n_{1},n_{2},k-1)\cdot\nu\cdot1_{\left\{ k>0\right\} }.\label{eq:DC-ARR-SER-bs-global-balance-equation}
\end{align}

Let $(X_{1},X_{2},Y)$ be a random vector that is distributed according to the stationary distribution $\pi$. 
So, $Y$ is distributed as the stationary distribution of the inventory process
in equilibrium and $X_{1}$ resp.~$X_{2}$ are respectively distributed as the queue length process
of priority resp.\ ordinary customers in equilibrium.

Consider the queueing-inventory system with lead time zero for the replenishment orders. Then $Y$ is constant $b$ and the distribution  of $(X_{1},X_{2})$ is the equilibrium of the 
priority system as described in Section \ref{sec:model} without inventory. A little reflection shows that the distribution of $X_1$ is geometrical with parameter $\lambda_{1}/\mu$, i.e.~the priority customers are served as in a standard $M/M/1/\infty$ with parameters $\lambda_{1}$ and $\mu$.
(Indeed, this observation applies to a classical $M/M/c$ queue with two priority classes under
a preemptive priority discipline as well. The priority customers behave as customers in a classic $M/M/c$ queue (cf.~\cite{wang;baron;scheller-wolf:15}).)

Our first proposition embeds a similar observation into the setting of the  model in Section \ref{sec:model}.
The problem with our queueing-inventory system is that the
two customer classes have to share the
same inventory and therefore, the ordinary (= non-priority) customers impose restrictions on the priority customers' service because they consume items from the inventory and generate more stock-outs experienced by the priority customers.

\begin{prop}
	\label{prop:GeometricPriorityCust}
	For the queueing-inventory system from Section \ref{sec:model} the stationary queue length of priority customers conditioned on positive inventory is geometrical with parameter $\lambda_{1}/\mu$, i.e.~for $n_{1}\in\mathbb{N}_{0}$ holds
	\begin{equation}
	P(X_{1}=n_{1}\mid Y>0)=P(X_{1}=0\mid Y>0)\cdot\left(\frac{\lambda_{1}}{\mu}\right)^{n_{1}},\label{eq:DC-ARR-SER-bs-condition}
	\end{equation}
	with normalisation constant
	\[
	P(X_{1}=0\mid Y>0) = 1-\frac{\lambda_{1}}{\mu}.
	\]
\end{prop}	
The proof of the proposition and some further properties presented below rely on  partial balance relations inherent in the global balance equations of $Z$.

\begin{lem}\label{lem:PartialBalance}
	For the queueing-inventory
	process $Z$ holds 
	\begin{align}
	& \phantomeq P(X_{1}=n_{1},Y>0)\cdot\lambda_{1}
	=P(X_{1}=n_{1}+1,Y>0)\cdot\mu, 
	\qquad\quad   n_{1}\in\mathbb{N}_{0},\label{eq:DC-ARR-SER-bs-prop-flows-x1}\\
	\nonumber \\
	& \phantomeq P(X_{2}=n_{2},0<Y\leq s)\cdot p\,\lambda_{2}+P(X_{2}=n_{2},Y>s)\cdot\lambda_{2}\nonumber \\
	& =P(X_{1}=0,X_{2}=n_{2}+1,Y>0)\cdot\mu, 
	\qquad \qquad \qquad \qquad \qquad n_{2}\in\mathbb{N}_{0},\label{eq:DC-ARR-SER-bs-prop-flows-x2}\\
	\nonumber \\
	& \phantomeq P(X_{1}+X_{2}=n,0<Y\leq s)\cdot(\lambda_{1}+p\:\lambda_{2})\nonumber \\
	& \phantomeq +P(X_{1}+X_{2}=n,Y>s)\cdot(\lambda_{1}+\lambda_{2})\nonumber \\
	& =P(X_{1}+X_{2}=n+1,0<Y\leq b)\cdot\mu, 
	\qquad \qquad \qquad \qquad \quad n\in\mathbb{N}_{0}.\label{eq:DC-ARR-SER-bs-prop-flows-x1+x2}
	\end{align}
\end{lem}	

All the equations can be proven by applying the cut-criterion for positive recurrent
processes (see \cite[Lemma 1.4, p.\ 8]{kelly:79}). This is
\begin{lem}\label{lem:CutCriterion}
	For ergodic $Z$ with stationary distribution $\pi$ it holds for any non-empty proper subset $A\subset E$
	\[
	\sum_{z\in A} \sum_{\tilde{z}\in A^c}\pi(z) q(z,\tilde{z}) =
	\sum_{\tilde{z}\in A^c} \sum_{z\in A} \pi(\tilde{z}) q(\tilde{z},z)
	\]
\end{lem}	
\begin{proof}[Proof of Lemma \ref{lem:PartialBalance}]
	For $n_{1}\in\mathbb{N}_{0}$, equation \prettyref{eq:DC-ARR-SER-bs-prop-flows-x1}
	can be proven by a cut, which divides $E$ into complementary sets
	according to the queue length of priority customers that is less than
	or equal to $n_{1}$ or greater than $n_{1}$, i.e.\ into the sets
	\[
	A:= \Big\{(m_{1},m_{2},k):m_{1}\in\{0,1,\ldots.n_{1}\},\:m_{2}\in\mathbb{N}_{0},\:k\in\{0,\ldots,b\}\Big\},\hphantom{\quad n_{1}\in\mathbb{N}_{0}.}
	\]
	and $A^c$.
	Then 
	for $n_{1}\in\mathbb{N}_{0}$ it holds by direct evaluation
	\begin{align*} 
	&\phantomeq\underbrace{\sum_{m_{1}=n_{1}}^{n_{1}}\sum_{m_{2}=0}^{\infty}\sum_{k=1}^{b}\pi(m_{1},m_{2},k)\cdot\lambda_{1}}_{=P(X_{1}=n_{1},Y>0)\cdot\lambda_{1}}=\underbrace{\sum_{\widetilde{m}_{1}=n_{1}+1}^{n_{1}+1}\sum_{\widetilde{m}_{2}=0}^{\infty}\sum_{\widetilde{k}=1}^{b}\pi(\widetilde{m}_{1},\widetilde{m}_{2},\widetilde{k})\cdot\mu}_{=P(X_{1}=n_{1}+1,Y>0)\cdot\mu}.
	\end{align*}
	Hence, for $n_{1}\in\mathbb{N}_{0}$ it holds \prettyref{eq:DC-ARR-SER-bs-prop-flows-x1}.\\
	For $n_{2}\in\mathbb{N}_{0}$, equation \prettyref{eq:DC-ARR-SER-bs-prop-flows-x2}
	can be proven by a cut, which divides $E$ into complementary sets
	according to the queue length of ordinary customers that is less than
	or equal to $n_{2}$ or greater than $n_{2}$, i.e.\ into the sets
	\[ A =
	\Big\{(m_{1},m_{2},k):m_{1}\in\mathbb{N}_{0},m_{2}\in\{0,1,\ldots.n_{2}\},\:k\in\{0,\ldots,b\}\Big\},\hphantom{\quad n_{2}\in\mathbb{N}_{0}.}
	\]
	and $A^c$.
	Then for $n_{2}\in\mathbb{N}_{0}$ it holds by direct evaluation
	\begin{align*}  
	&\phantomeq\underbrace{\sum_{m_{1}=0}^{\infty}\sum_{m_{2}=n_{2}}^{n_{2}}\sum_{k=1}^{s}\pi(m_{1},m_{2},k)\cdot p\,\lambda_{2}}_{=P(X_{2}=n_{2},0<Y\leq s)\cdot p\,\lambda_{2}}+\underbrace{\sum_{m_{1}=0}^{\infty}\sum_{m_{2}=n_{2}}^{n_{2}}\sum_{k=s+1}^{b}\pi(m_{1},m_{2},k)\cdot\lambda_{2}}_{=P(X_{2}=n_{2},Y>s)\cdot\lambda_{2}}\\
	& =\underbrace{\sum_{\widetilde{m}_{1}=0}^{0}\sum_{\widetilde{m}_{2}=n_{2}+1}^{n_{2}+1}\sum_{\widetilde{k}=1}^{b}\pi(\widetilde{m}_{1},\widetilde{m}_{2},\widetilde{k})\cdot\mu}_{=P(X_{1}=0,X_{2}=n_{2}+1,Y>0)\cdot\mu}.
	\end{align*}
	Thus, for $n_{2}\in\mathbb{N}_{0}$ it holds \prettyref{eq:DC-ARR-SER-bs-prop-flows-x2}.\\
	For $n\in\mathbb{N}_{0}$, equation \prettyref{eq:DC-ARR-SER-bs-prop-flows-x1+x2}
	can be proven by a cut, which divides $E$ into complementary sets
	according to the size of the total queue length that is less than
	or equal to $n$ or greater than $n$, i.e.\ into the sets
	\[A=
	\Big\{(m_{1},m_{2},k):m_{1}\in\mathbb{N}_{0},\:m_{2}\in\mathbb{N}_{0},\:(m_{1}+m_{2})\in\{0,1,\ldots,n\},\:k\in\{0,\ldots,b\}\Big\},\hphantom{\ n\in\mathbb{N}_{0}.}
	\]
	and $A^c$.
	Then for $n\in\mathbb{N}_{0}$ holds by direct evaluation
	\begin{align*}
	&   \phantomeq\underbrace{\sum_{m_{1}+m_{2}=n}^{n}\sum_{k=1}^{s}\pi(m_{1},m_{2},k)\cdot(\lambda_{1}+p\,\lambda_{2})}_{=P(X_{1}+X_{2}=n,0<Y\leq s)\cdot(\lambda_{1}+p\:\lambda_{2})}+\underbrace{\sum_{m_{1}+m_{2}=n}^{n}\sum_{k=s+1}^{b}\pi(m_{1},m_{2},k)\cdot(\lambda_{1}+\lambda_{2})}_{=P(X_{1}+X_{2}=n,Y>s)\cdot(\lambda_{1}+\lambda_{2})}\\
	& =\underbrace{\sum_{\widetilde{m}_{1}+\widetilde{m}_{2}=n+1}^{n+1}\sum_{\widetilde{k}=1}^{b}\pi(\widetilde{m}_{1},\widetilde{m}_{2},\widetilde{k})\cdot\mu}_{=P(X_{1}+X_{2}=n+1,Y>0)\cdot\mu}.
	\end{align*}
	Thus, for $n\in\mathbb{N}_{0}$ it holds \prettyref{eq:DC-ARR-SER-bs-prop-flows-x1+x2}.	
\end{proof}	

\begin{proof}[Proof of \textbf{Proposition \ref{prop:GeometricPriorityCust}}]
	From \prettyref{eq:DC-ARR-SER-bs-prop-flows-x1} it follows
	by iteration for $n_{1}\in\mathbb{N}_{0}$
	\begin{equation}
	P(X_{1}=n_{1},Y>0)=P(X_{1}=0,Y>0)\cdot\left(\frac{\lambda_{1}}{\mu}\right)^{n_{1}}.
	\label{eq:DC-ARR-SER-bs-condition}
	\end{equation}
	Conditioning and exploiting $\lambda_{1}<\mu$ for summation ends the proof.
\end{proof}

A direct consequence of Lemma \ref{lem:PartialBalance} are the following intuitive rate equations which equalize the mean admitted arrivals per customer classes to class dependent throughputs (= effective departure rates).
The proof is in any case by direct summation of the respective formulas in Lemma \ref{lem:PartialBalance}.
\begin{prop}
	\label{prop:DC-ARR-SER-bs-prop-flows} For the queueing-inventory
	process it holds the following equilibrium of probability flows
	\begin{equation}
	\underbrace{P(Y>0)\cdot\lambda_{1}}_{\ensuremath{\substack{\text{effective arrival rate}\\
				\text{of priority customers}
			}
	}}=\underbrace{P(X_{1}>0,Y>0)\cdot\mu}_{\ensuremath{\substack{\text{effective departure rate}\\
				\text{of priority customers}
			}
	}},\label{eq:DC-ARR-SER-bs-prob-flows-eq1}
	\end{equation}
	\vspace{0.3cm}
	\begin{equation}
	\underbrace{P(0<Y\leq s)\cdot p\,\lambda_{2}+P(Y>s)\cdot\lambda_{2}}_{\ensuremath{\substack{\text{effective arrival rate}\\
				\text{of ordinary customers}
			}
	}}=\underbrace{P(X_{1}=0,X_{2}>0,Y>0)\cdot\mu}_{\ensuremath{\substack{\text{effective departure rate}\\
				\text{of ordinary customers}
			}
	}},\label{eq:DC-ARR-SER-bs-prob-flows-eq3}
	\end{equation}
	\begin{align}
	& \phantomeq\underbrace{P(Y>0)\cdot\lambda_{1}+P(0<Y\leq s)\cdot p\,\lambda_{2}+P(Y>s)\cdot\lambda_{2}}_{\substack{\text{effective arrival rate}\\
			\text{of customers}
		}
	},\nonumber \\
	& =\underbrace{P(X_{1}+X_{2}>0,Y>0)\cdot\mu}_{\substack{\text{effective departure rate}\\
			\text{of customers}
		}
	}\label{eq:DC-ARR-SER-bs-prop-Nn-Ygr0-1}
	\end{align}
\end{prop}

Similar to the flow equations in Lemma \ref{lem:PartialBalance}
it is possible to derive flow equations with respect to the inventory level. The derivation is more tediuos and can be found in \cite[Proposition 11.1.6]{otten:18}.
The main result is that for $k=0,1,\dots,b-1$ it holds 
\begin{equation}\label{eq:PBInventory}
P(Y=k)\cdot \nu = P(Y=k+1, X_1 + X_2 >0)\cdot \mu
\end{equation}

Note that the statement of \prettyref{eq:PBInventory} 
exhibits an insensitivity
property with respect to variation of the parameters of the system,
more specifically it is independent of the threshold level $s$ and the arrival intensities $\lambda_{1}$ and $\lambda_{2}$.

$ $
\begin{rem}\label{rem:C>2}
	\textbf{(1)}	The results in this section can be generalised in a direct way to
	the case of a system with $C$ customer classes, where $\overline{C}=\{1,\ldots,C\}$
	is the set of customer classes. Customers of type $c$  arrive with rate $\lambda_{c}>0$,
	a priority parameter $p_{c}$ ($0\leq p_{c}\leq1$) and a threshold
	level $s_{c}$, $c\in\overline{C}$.\\
	\textbf{(2)}	
	In the models of Isotupa~\cite{isotupa:15},
	the replenishment rate depends on the number of pending orders. 	
	We can extend our model so that the replenishment lead time depends
	on the number of orders at the supplier. If there are $k>0$ orders
	present at the supplier, the intensity of the replenishment lead time
	is $\nu(k)>0$. For more details see \cite[Chapter 11.1]{otten:18}.
\end{rem}

\begin{proof}[Proof of \textbf{Theorem \ref{thm:p=1}}]
	Recall that the global balance equations for $Z$ usually are stated with unknown positive $x=(x(z):z\in E$) and that, in case of irreducibility (under $Q$) of $E$,  $Z$ is ergodic if and only if there exist a strictly positive solution $x$ which is summable, i.e.~$\sum_{z\in E} x(z)<\infty$.
	The first observation is that the partial balance relations	
	of Lemma \ref{lem:PartialBalance}
	hold for any solution $x$, even if the solution is not summable. We shall exploit equation \eqref{eq:DC-ARR-SER-bs-prop-flows-x1+x2} which reads in case of $p=1$ for a general solution 
	\begin{eqnarray*}
		&&	\sum_{n_{1}+n_{2}=n}\left(\sum_{k=1}^b x(n_1,n_2,k)\right)\cdot(\lambda_{1}+\lambda_{2})
		=
		\sum_{n_{1}+n_{2}=n+1}\left(\sum_{k=1}^b x(n_1,n_2,k)\right)\cdot\mu,\qquad
		n\in\mathbb{N}_0 \,. 
	\end{eqnarray*}
	It follows for $n\in\mathbb{N}_0$
	\begin{equation*}
	\sum_{n_{1}+n_{2}=n}\left(\sum_{k=1}^b x(n_1,n_2,k)\right)
	=	
	\left(\sum_{k=1}^b x(0,0,k)\right)\cdot \left(\frac{\lambda_{1}+\lambda_{2}}{\mu}\right)^n.
	\end{equation*}
	Assuming ergodicity of $Z$ the
	summability condition yields
	\begin{eqnarray*}
		&&\infty >	\sum_{n_{1}=0}^\infty \sum_{n_{2}=0}^\infty
		\sum_{k=0}^b  x(n_1,n_2,0)\\
		&=&\sum_{n_{1}=0}^\infty \sum_{n_{2}=0}^\infty
		\left(\sum_{k=1}^b x(n_1,n_2,k)\right)
		+
		\sum_{n_{1}=0}^\infty \sum_{n_{2}=0}^\infty x(n_1,n_2,0)\\
		&=&\sum_{n=0}^\infty\left\{ \sum_{n_1+n_{2}=n}
		\left(\sum_{k=1}^b x(n_1,n_2,k)\right)\right\}
		+
		\sum_{n_1=0}^\infty \sum_{n_{2}=0}^\infty x(n_1,n_2,0)\\
		&=&\sum_{n=0}^\infty\left\{ \left(\sum_{k=1}^b x(0,0,k)\right)\cdot \left(
		\frac{\lambda_{1}+\lambda_{2}}{\mu}\right)^n
		\right\}
		+\underbrace{\sum_{n_1=0}^\infty \sum_{n_{2}=n}^\infty x(n_1,n_2,0)}_{\in (0,\infty)}.
	\end{eqnarray*}
	Because of ${\sum_{k=1}^b x(0,0,k)\in (0,\infty)}$ the last term is finite if and only if $\lambda_{1}+\lambda_{2}<\mu$
	holds. This finishes the proof. 	
\end{proof}


\section{The case of instant service \label{sec:DC-bs-pure-inventory}}
In this section, we consider the case of instant service (zero production time), which turns 
the model into the formalism of classical inventory theory. The
supply chain of interest is depicted in \prettyref{fig:DC-bs-pure-inventory-1}
and consists of priority and ordinary customers, an inventory and a supplier.

\begin{figure}[h]
	\centering{}\includegraphics[width=0.75\textwidth]{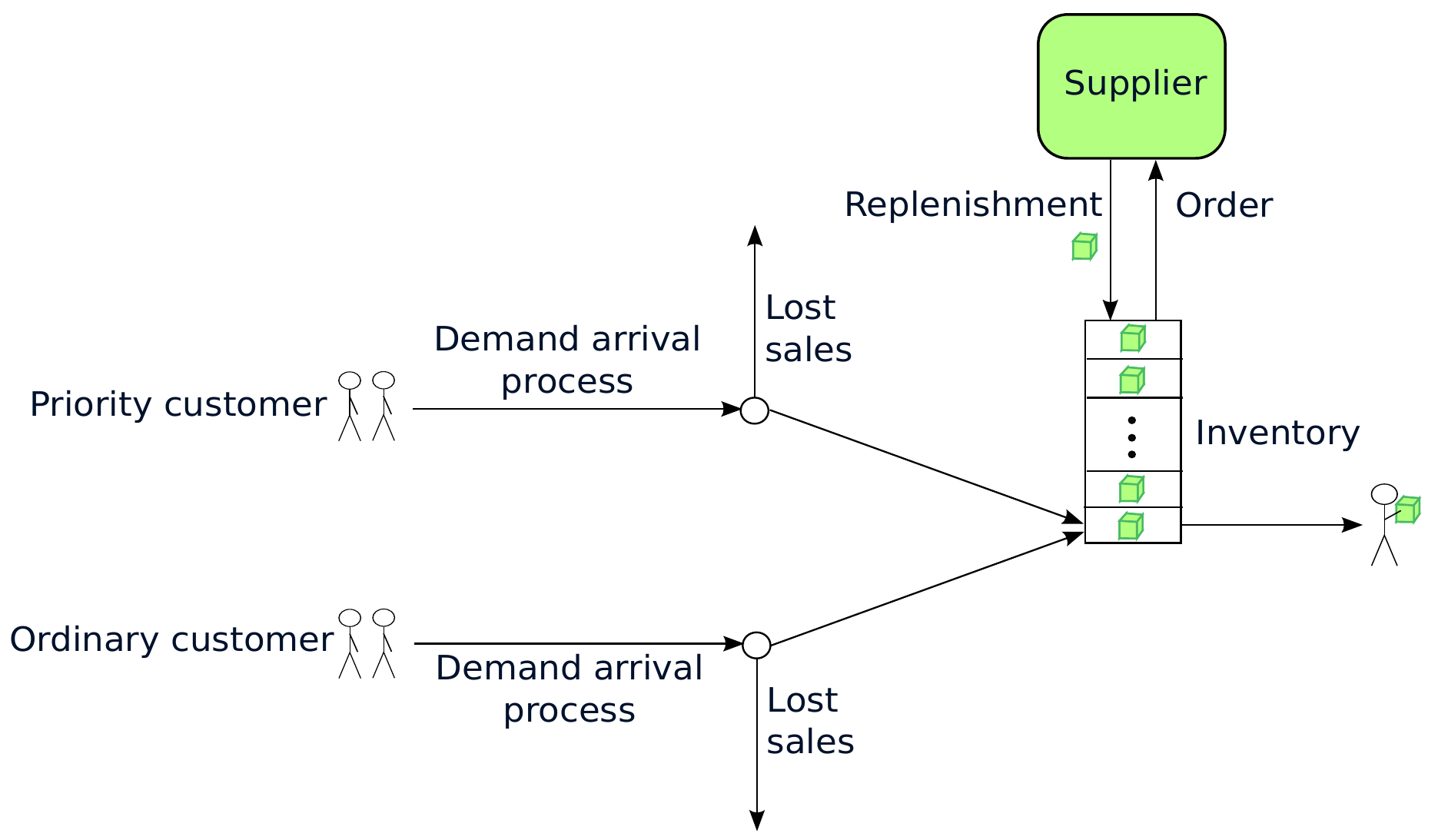}\caption{\label{fig:DC-bs-pure-inventory-1}The pure inventory system with
		two customer classes}
\end{figure}

A short reflection shows that due to the lost sales assumption no customer queues will arise. Therefore, the on-hand inventory process $Y=(Y(t):t\geq0)$ carries all information for a Markovian description of the system,
which is a slight generalisation of the pure inventory model in \cite{isotupa:15}. Isotupa derives the stationary distribution but in her model the priority parameter $p$ is missing.  Liu and coauthors (\cite{cheng;zhou;liu:12},~\cite{liu;xi;chen:13},~\cite{liu;feng;wong:14}) determine the stationary distribution for a pure inventory model with priority parameter but with a different replenishment policy.

The state space of $Y$ is
\[
K=\{0,\ldots,b\},
\]
where $b$ is the maximal size of the inventory.
$Y$ is irreducible and therefore ergodic
and we define the limiting and stationary distribution of $Y$ by
\[
\theta:=\left(\theta\left(k\right):k\in K\right),\quad\theta\left(k\right):=\lim_{t\rightarrow\infty}P\left(Y(t)=k\right).
\]
$\theta$ satisfies the following global balance equations in case of $s<b$.
\begin{align*}
& \theta(0)\cdot\nu=\theta(1)\cdot(\lambda_{1}+p\,\lambda_{2}),\\
& \theta(k)\cdot(\lambda_{1}+p\,\lambda_{2}+\nu)=\theta(k+1)\cdot(\lambda_{1}+p\,\lambda_{2})+\theta(k-1)\cdot\nu, &  & k=1,\ldots,s-1,\\
& \theta(s)\cdot(\lambda_{1}+p\,\lambda_{2}+\nu)=\theta(s+1)\cdot(\lambda_{1}+\lambda_{2})+\theta(s-1)\cdot\nu,\\
& \theta(k)\cdot(\lambda_{1}+\lambda_{2}+\nu)=\theta(k+1)\cdot(\lambda_{1}+\lambda_{2})+\theta(k-1)\cdot\nu, &  & k=s+1,\ldots,b-1,\\
& \theta(b)\cdot(\lambda_{1}+\lambda_{2})=\theta(b-1)\cdot\nu.
\end{align*}
It is direct to see that these are the balance equations for a finite  birth-death process and we have immediately the following result.
\begin{prop}
	The inventory process $Y=(Y(t):t\geq0)$ has the following limiting
	and stationary distribution
	\begin{align*}
	\theta(0) & =\left[\sum_{j=0}^{s}\left(\frac{\nu}{\lambda_{1}+p\,\lambda_{2}}\right)^{j}+\left(\frac{\nu}{\lambda_{1}+p\,\lambda_{2}}\right)^{s}\cdot\sum_{j=1}^{b-s}\left(\frac{\nu}{\lambda_{1}+\lambda_{2}}\right)^{j}\right]^{-1}\\
	& =\left[\frac{1-\left(\frac{\nu}{\lambda_{1}+p\,\lambda_{2}}\right)^{s+1}}{1-\left(\frac{\nu}{\lambda_{1}+p\,\lambda_{2}}\right)}+\left(\frac{\nu}{\lambda_{1}+p\,\lambda_{2}}\right)^{s}\cdot\left(\frac{1-\left(\frac{\nu}{\lambda_{1}+\lambda_{2}}\right)^{b-s+1}}{1-\left(\frac{\nu}{\lambda_{1}+\lambda_{2}}\right)}-1\right)\right]^{-1}\\
	& =\left[\frac{1-\left(\frac{\nu}{\lambda_{1}+p\,\lambda_{2}}\right)^{s+1}}{1-\left(\frac{\nu}{\lambda_{1}+p\,\lambda_{2}}\right)}+\left(\frac{\nu}{\lambda_{1}+p\,\lambda_{2}}\right)^{s}\cdot\left(\frac{\nu}{\lambda_{1}+\lambda_{2}}\right)\cdot\left(\frac{1-\left(\frac{\nu}{\lambda_{1}+\lambda_{2}}\right)^{b-s}}{1-\left(\frac{\nu}{\lambda_{1}+\lambda_{2}}\right)}\right)\right]^{-1},\\
	\theta(k) & =\left(\frac{\nu}{\lambda_{1}+p\,\lambda_{2}}\right)^{k}\cdot\theta(0),\qquad\qquad\qquad\qquad\ k=1,\ldots,s,\\
	\theta(k) & =\left(\frac{\nu}{\lambda_{1}+p\,\lambda_{2}}\right)^{s}\left(\frac{\nu}{\lambda_{1}+\lambda_{2}}\right)^{k-s}\cdot\theta(0),\qquad k=s+1,\ldots,b.
	\end{align*}
\end{prop}

\newpage
\section{Conclusion}\label{sect:Conclusion}
We invenstigated ergodicity for queueing-inventory systems with two priority classes in case of unbounded queues for both
customer classes. The admission control to the queues is flexible and incorporates two parameters $(s, p)$ which mainly restrict entrance of low priority customers. 
The main result is a stability condition which is proved by construction of a suitable Lyapunov function. Although we proved that in case of parameter constellation $(s, 1)$ the condition is sufficient and necessary we believe that in general the condition is not sharp.
To clarify the situation in general is part of our ongoing research.

\bibliographystyle{alpha}
\bibliography{sn-bibliography}

\newcommand{\etalchar}[1]{$^{#1}$}
\begin{thebibliography}{WBSW15}

\bibitem[AGR07]{arslan;graves;roemer:07}
H.~Arslan, S.C. Graves, and T.A. Roemer.
\newblock A single-product inventory model for multiple demand classes.
\newblock {\em Management Science}, 53(9):1486--1500, 2007.

\bibitem[BDK17]{baek;dudina;kim:17}
J.W. Baek, O.~S. Dudina, and C.~S. Kim.
\newblock A queueing system with heterogeneous impatient customers and
  consumable additional items.
\newblock {\em International Journal of Applied Mathematics and Computer
  Science}, 27:367 -- 384, 2017.

\bibitem[BV11]{bijvank;vis:11}
M.~Bijvank and F.A. Vis.
\newblock Lost-sales inventory theory: A review.
\newblock {\em European Journal of Operations Research}, 215:1--13, 2011.

\bibitem[CZL12]{cheng;zhou;liu:12}
H.~Chen, Z.~Zhou, and M.~Liu.
\newblock A queueing-inventory system with two classes demand and subject to
  selected service.
\newblock {\em Journal of Information \& Computational Science},
  9(11):3081--3089, 2012.

\bibitem[Dad90]{daduna:90}
H.~Daduna.
\newblock Exchangeable {I}tems in {R}epair {S}ystems: {D}elay {T}imes.
\newblock {\em Operations Research}, 38(2):349--354, 1990.

\bibitem[Dim22]{dimitriou:22}
I.~Dimitriou.
\newblock Stationary analysis of certain {Markov}-modulated reflected random
  walks in the quarter plane.
\newblock {\em Annals of Operations Research}, 310:355--387, 2022.

\bibitem[FIM99]{fayolle;iasnogorodski;malyshev:99}
G.~Fayolle, R.~Iasnogorodski, and V.~Malyshev.
\newblock {\em Random walks in the quarter-plane}, volume~40 of {\em
  Applications of Mathematics}.
\newblock Springer, Berlin, Heidelberg, New York, 1999.

\bibitem[Fos53]{foster:57}
F.~G. Foster.
\newblock On the stochastic matrices associated with certain queuing processes.
\newblock {\em The Annals of Mathematical Statistics}, 24(3):355--360, 1953.

\bibitem[GCB02]{gruen;corsten;bharadwaj:2002}
T.W. Gruen, D.~Corsten, and S.~Bharadwaj.
\newblock {\em Retail Out-of-Stocks: A Worldwide Examination of Extent Causes
  and Consumer Responses}.
\newblock Grocery Manufacturers of America, Washington, 2002.

\bibitem[IS13]{isotupa;Samanta:13}
K.P.S. Isotupa and S.K. Samanta.
\newblock A continuous review $(s,{Q})$ inventory system with priority
  customers and arbitrarily distributed lead times.
\newblock {\em Mathematical and Computer Modelling}, 57:1259--1269, March 2013.

\bibitem[Iso06]{isotupa:06b}
K.P.S. Isotupa.
\newblock An $(s,{Q})$ {M}arkovian inventory system with lost sales and two
  demand classes.
\newblock {\em Mathematical and Computer Modelling}, 43:687--694, 2006.

\bibitem[Iso07]{isotupa:07sapna07}
K.P.S. Isotupa.
\newblock Continuous {R}eview $(s,{Q})$ {I}nventory {S}ystem with {T}wo {T}ypes
  of {C}ustomers.
\newblock {\em International Journal of Agile Manufacturing}, 9(1):79--85,
  2007.

\bibitem[Iso11]{isotupa:11}
K.P.S. Isotupa.
\newblock An $(s,{Q})$ inventory system with two classes of customers.
\newblock {\em International Journal of Operational Research}, 12:1--19, 2011.

\bibitem[Iso15]{isotupa:15}
K.P.S. Isotupa.
\newblock Cost analysis of an $({S}-1,{S})$ inventory system with two demand
  classes and rationing.
\newblock {\em Annals of Operations Research}, 233:411--421, January 2015.

\bibitem[Jai68]{jaiswal:68}
N.K. Jaiswal.
\newblock {\em Priority Queues}, volume~50 of {\em Mathematics in science and
  engineering}.
\newblock Academic Press, New York -- London, 1968.

\bibitem[JAK13]{jeganathan;anbazhagan;kathiresan:13}
K.~Jeganathan, N.~Anbazhagan, and J.~Kathiresan.
\newblock A retrial inventory system with non-preemptive priority service.
\newblock {\em International Journal of Information and Managment Sciences},
  24(1):57--77, 2013.

\bibitem[Jeg15]{jeganathan:15}
K.~Jeganathan.
\newblock Linear retrial inventory system with second optional service under
  mixed priority service.
\newblock {\em TWMS Journal of Applied and Engineering Mathematics},
  5(2):249--268, 2015.

\bibitem[JKA16]{jeganathan;kathiresan;anbazhagan:16}
K.~Jeganathan, J.~Kathiresan, and N.~Anbazhagan.
\newblock A retrial inventory system with priority customers and second
  optional service.
\newblock {\em OPSEARCH}, 53:808--834, 2016.

\bibitem[Kap70]{kaplan:70}
R.~S. Kaplan.
\newblock A dynamic inventory model with stochastic lead times.
\newblock {\em Management Science}, 16(7):491--507, 1970.

\bibitem[Kar57]{karush:57}
W.~Karush.
\newblock A queuing model for an inventory problem.
\newblock {\em Operations Research}, 5(5):693--703, 1957.

\bibitem[Kel79]{kelly:79}
F.P. Kelly.
\newblock {\em Reversibility and Stochastic Networks}.
\newblock John Wiley and Sons, Chichester -- New York -- Brisbane -- Toronto,
  1979.

\bibitem[KSN21]{krishnamoorthy;shajin;narayanan:21}
A.~Krishnamoorthy, D.~Shajin, and V.C. Narayanan.
\newblock Inventory with positive service time: a survey.
\newblock In V.~Anisimov and N.~Limnios, editors, {\em Queueing Theory 2},
  pages 201--237, London, 2021. Wiley.

\bibitem[KY14]{kelly;yudovina:14}
F.P. Kelly and E.~Yudovina.
\newblock {\em Stochastic {N}etworks}.
\newblock Cambridge University Press, Cambridge, 2014.

\bibitem[LAR21]{li;arreola-risa:21}
B.~Li and A.~Arreola-Risa.
\newblock On minimizing downside risk in make-to-stock, risk averse firms.
\newblock {\em Naval Reseearch Logistics}, 68:199--213, 2021.

\bibitem[LFW14]{liu;feng;wong:14}
M.~Liu, M.~Feng, and C.Y. Wong.
\newblock Flexible service policies for a {M}arkov inventory system with two
  demand classes.
\newblock {\em International Journal of Production Economics}, 151:180--185,
  2014.

\bibitem[LR99]{latouche;ramaswami:99}
G.~Latouche and V.~Ramaswami, editors.
\newblock {\em Introduction to {M}atrix {A}nalytic {M}ethods in {S}tochastic
  {M}odeling}.
\newblock ASA-SIAM Series on Statistics and Applied Probability. SIAM,
  Philadelphia, 1999.

\bibitem[LXC13]{liu;xi;chen:13}
M.~Liu, F.~Xi, and H.~Chen.
\newblock Control policies for a {M}arkov queueing-inventory system with two
  demand classes.
\newblock In {\em International Asia Conference on Industrial Engineering and
  Management Innovation (IEMI2012) Proceedings}, pages 1543--1550, 2013.

\bibitem[LZ92]{lee;zipkin:92}
Y.J. Lee and P.~Zipkin.
\newblock Tandem queues with planned inventories.
\newblock {\em Operations Research}, 40:936--947, 1992.

\bibitem[LZ95]{lee;zipkin:95}
Y.J. Lee and P.~Zipkin.
\newblock Processing networks with inventories: {S}equential refinement
  systems.
\newblock {\em Operations Research}, 43:1025--1036, 1995.

\bibitem[LZ09]{li;zhao:09}
H.~Li and Y.Q. Zhao.
\newblock Exact tail asymptotics in a priority queue - characterizations of the
  preemptive model.
\newblock {\em Queueing Systems}, 63:355--381, 2009.

\bibitem[MF98]{melikov;fatalieva:98}
A.~Z. Melikov and M.R. Fatalieva.
\newblock Situational inventory in counter-stream serving systems.
\newblock {\em Engineering Simulation}, 15:839 -- 848, 1998.

\bibitem[Mil58]{miller:58}
R.G. Miller.
\newblock Priority queues.
\newblock Technical report, Stanford University, California, 1958.

\bibitem[MM92]{melikov;molchanov:92}
A.~Z. Melikov and A.~A. Molchanov.
\newblock Stock optimization in transportation/storage systems.
\newblock {\em Cybernetics and Systems Analysis}, 28(3):484 -- 487, 1992.

\bibitem[Mor58]{morse:1958}
P.M. Morse.
\newblock {\em Queues, Inventories and Maintenance}.
\newblock Wiley, New York, 1958.

\bibitem[MPA18a]{melikov;ponomarenko;aliyev:18a}
A.~Z. Melikov, L.~A. Ponomarenlo, and I.~A. Aliyev.
\newblock Analysis and optimization of models of queueing-inventory systems
  with two types of requests.
\newblock {\em Journal of Automation and Information Sciences}, 50(12):34--50,
  2018.

\bibitem[MPA18b]{melikov;ponomarenko;aliyev:18}
A.~Z. Melikov, L.~A. Ponomarenlo, and I.~A. Aliyev.
\newblock Markov models of systems with demands of two types and different
  restocking policies.
\newblock {\em Cybernetics and Systems Analysis}, 54(6):900--917, 2018.

\bibitem[Ott18]{otten:18}
S.~Otten.
\newblock {\em Integrated models for performance analysis and optimization of
  queueing-inventory systems in logistics networks}.
\newblock PhD thesis, Universit{\"a}t Hamburg, Department of Mathematics,
  Hamburg, Germany, 2018.

\bibitem[RBP13]{ravid;boxma;perry:13}
R.~Ravid, O.J. Boxma, and D.~Perry.
\newblock Repair systems with exchangeable items and the longest queue
  mechanism.
\newblock {\em Queueing Systems}, pages 295--316, 2013.

\bibitem[RM11]{rego:11}
J.R.~do Rego and M.A.~de Mesquita.
\newblock Spare parts inventory control: a literature review.
\newblock {\em Produ\c{c}\~{a}o}, 21(4):656--666, 2011.

\bibitem[RW96]{rubio;wein:96}
R.~Rubio and L.M. Wein.
\newblock Setting base stock levels using product-form queueing networks.
\newblock {\em Management Science}, 42:259--268, 1996.

\bibitem[RZ17]{reed;zhang:17}
J.~Reed and B.~Zhang.
\newblock Managing capacity and inventory for multi-server make-to-stock
  queues.
\newblock {\em Queueing Systems}, 86:61--94, 2017.

\bibitem[SDDK20]{shajin;dudin;dudina;krishnamoorthy:20}
D.~Shajin, A.N. Dudin, O.S. Dudina, and A.~Krishnamoorthy.
\newblock A two-priority single server retrial queue with additional items.
\newblock {\em Journal of Industrial \& Management Optimization}, 16(6):2891
  --2912, 2020.

\bibitem[SSD{\etalchar{+}}06]{schwarz;sauer;daduna;kulik;szekli:06}
M.~Schwarz, C.~Sauer, H.~Daduna, R.~Kulik, and R.~Szekli.
\newblock ${M/M/1}$ queueing systems with inventory.
\newblock {\em Queueing Systems: Theory and Applications}, 54:55--78, 2006.

\bibitem[SSL92]{sigman;simchi-levi:92}
K.~Sigman and D.~Simchi-Levi.
\newblock Light traffic heuristic for an {M}/{G}/1 queue with limited
  inventory.
\newblock {\em Annals of Operations Research}, 40:371 -- 380, 1992.

\bibitem[Tem05]{tempelmeier:2005}
H.~Tempelmeier.
\newblock {\em Bestandsmanagement in Supply Chains}.
\newblock Books on Demand, Norderstedt, 2005.

\bibitem[VS06]{verhoef;sloot:2006}
P.~Verhoef and L.M. Sloot.
\newblock Out-of-stock: Reactions, antecedents, management solutions, and a
  future perspective.
\newblock In M.~Krafft and M.K. Mantrala, editors, {\em Retailing in the 21st
  Century: Current and Future Trends}, pages 239--253. Springer, Philadelphia,
  Pennsylvania, USA, 2006.

\bibitem[Wan15]{wang:15}
F.-F. Wang.
\newblock Approximation and optimization of a multi-server impatient retrial
  inventory-queueing system with two demand classes.
\newblock {\em Quality Technology \& Quantitative Management}, 12(3):269--292,
  2015.

\bibitem[WBSW15]{wang;baron;scheller-wolf:15}
J.~Wang, O.~Baron, and A.~Scheller-Wolf.
\newblock ${M/M/c}$ {Q}ueue with {T}wo {P}riority {C}lasses.
\newblock {\em Operations Research}, 63(3):733--749, 2015.

\bibitem[Wol89]{wolff:89}
R.W. Wolff.
\newblock {\em Stochastic Modeling and the Theory of Queues}.
\newblock Prentice-Hall International Editions, Englewood Cliffs, 1989.

\bibitem[YAJ15]{yadavalli;anbazhagan;jeganathan:15}
V.S.S. Yadavalli, N.~Anbazhagan, and K.~Jeganathan.
\newblock A retrial inventory system with impatient customers.
\newblock {\em Applied Mathematics $\&$ Information Sciences}, 9(2):637--650,
  2015.

\bibitem[Zaz94]{zazanis:94}
M.~Zazanis.
\newblock Push and {P}ull {S}ystems with {E}xternal {D}emands.
\newblock In {\em Proceedings of the 32nd Allerton Conference on Communication,
  Control, and Computing}, Allerton, Illinois, 1994.

\bibitem[ZL11]{zhao;lian:11}
N.~Zhao and Z.~Lian.
\newblock {A queueing-inventory system with two classes of customers}.
\newblock {\em International Journal of Production Economics}, 129(1):225--231,
  2011.

\end{thebibliography}
	
\end{document}